\documentclass[letter,12pt]{article}

\makeatletter
\@addtoreset{footnote}{page}
\makeatother

\usepackage{style-enumitem}

\renewcommand{\labelenumi}{(\theenumi)}

\usepackage{amsthm}
\usepackage{amsmath,amssymb,latexsym,amsfonts,mathrsfs}

\renewcommand{\hat}{\widehat}

\renewcommand{\bar}{\overline}

\newcommand{\absol}[1]{\left| #1 \right|} 
\newcommand{\rbra}[1]{\!\left( #1 \right)} 
\newcommand{\cbra}[1]{\!\left\{ #1 \right\}} 
\newcommand{\sbra}[1]{\!\left[ #1 \right]} 

\newcommand{\bE}{\ensuremath{\mathbb{E}}}
\newcommand{\bF}{\ensuremath{\mathbb{F}}}

\newcommand{\bN}{\ensuremath{\mathbb{N}}}

\newcommand{\bP}{\ensuremath{\mathbb{P}}}

\newcommand{\bR}{\ensuremath{\mathbb{R}}}

\newcommand{\cA}{\ensuremath{\mathcal{A}}}

\newcommand{\cF}{\ensuremath{\mathcal{F}}}

\theoremstyle{plain}
\newtheorem{Thm}{Theorem}[section]

\theoremstyle{definition}
\newtheorem{Ass}[Thm]{Assumption}

\numberwithin{equation}{section}

\makeatletter
\renewcommand\section{\@startsection {section}{1}{\z@}%
                                   {-3.5ex \@plus -1ex \@minus -.2ex}%
                                   {2.3ex \@plus.2ex}%
                                   {\normalfont\large\bf}}
\makeatother

\makeatletter
\renewcommand\subsection{\@startsection {subsection}{1}{\z@}%
                                   {-3.5ex \@plus -1ex \@minus -.2ex}%
                                   {2.3ex \@plus.2ex}%
                                   {\normalfont\normalsize\bf}}
\makeatother

\usepackage[hmargin=0.8in,height=8.6in]{geometry}
\usepackage{ifpdf}
\ifpdf
\usepackage[pdftex]{graphicx}
 \else
\usepackage[dvips]{graphicx}
 \fi

\usepackage{color}
\usepackage{mathtools}
\mathtoolsset{showonlyrefs=true}

\definecolor{webgreen}{rgb}{0,.5,0}
\definecolor{webbrown}{rgb}{.8,0,0}
\definecolor{emphcolor}{rgb}{0.5,0.95,0.95}

\usepackage{hyperref}
\hypersetup{%
	colorlinks=true,
	linkcolor=webbrown,
	filecolor=webbrown,
	citecolor=webgreen,
	breaklinks=true}
\ifpdf \hypersetup{pdftex,
	pdfstartview=FitH, 
	bookmarksopen=true,
	bookmarksnumbered=true
} \else \hypersetup{dvips} \fi

\usepackage{color}

\numberwithin{equation}{section}
\newtheorem{problem}{Problem}[section]
\newtheorem{theorem}[problem]{Theorem}
\newtheorem{corollary}[problem]{Corollary}
\newtheorem{remark}[problem]{Remark}
\newtheorem{lemma}[problem]{Lemma}
\newcommand {\R}{\mathbb{R}}

\newcommand {\N}{\mathbb{N}}

\newcommand {\E}{\mathbb{E}}

\newcommand{\diff}{{\rm d}}

\newcommand{\lev}{L\'{e}vy }

\allowdisplaybreaks[4]

\begin{document}

\begin{center}
{\Large \bf 
Refraction strategies in stochastic control: optimality for a general L\'evy process model
}
\end{center}
\begin{center}
Kei Noba, Jos\'e Luis P\'erez and Kazutoshi Yamazaki
\end{center}

\begin{abstract}
We revisit an absolutely-continuous version of the stochastic control problem driven by a \lev process. 
 A strategy must be absolutely continuous with respect to the Lebesgue measure and the running cost function is assumed to be convex.  We show the optimality of a refraction strategy, which adjusts the drift of the state process at a constant rate whenever it surpasses a certain threshold. The optimality holds for a general \lev process, generalizing the spectrally negative case presented in \cite{HerPerYam2016}. \\
\noindent \small{\noindent  \textbf{MSCcodes}: 60G51, 93E20, 90B05 }
\end{abstract}

\section{Introduction}


In the last couple of decades, there has been great progress in the theory of stochastic control for non-Gaussian processes with jumps. 
In the continuous-time stochastic model, randomness is modeled in terms of a continuous-time stochastic process, called a state process. The objective is to optimally control the path of the state process to maximize or minimize the expected payoffs or costs that depend on the path of the realized controlled state process as well as the control process. A majority of the existing literature focus on the models driven by continuous Gaussian processes. However, in recent years, progress has been made for spectrally one-sided \lev processes (\lev processes with one-sided jumps) thanks to the development of the fluctuation theory and the scale function. 

In this paper, we revisit a version of continuous-time stochastic control under the absolutely continuous assumption. In the classical formulation of singular control, the set of admissible strategies consists of adapted processes of bounded variation, which includes those not implementable in practice.
 The absolutely-continuous model thus imposes an additional restriction requiring that a control process must have a bounded density with respect to the Lebesgue measure.

There is an interesting analogy between the classical singular control model and the absolutely continuous model. In the former, a barrier strategy, aka a reflection strategy, is a natural selection for a candidate optimal strategy. 
In the one-dimensional Brownian or \lev models, a barrier strategy reflects the state process at a suitable boundary in the classical (Skorokhod) sense and thus classical stochastic analysis methods can be directly applied to show its optimality. 
By contrast,  in the case with absolutely-continuous assumptions, where barrier strategies are no longer admissible, a so-called \emph{threshold strategy}, or a \emph{refraction strategy}, becomes a candidate optimal strategy.
It pushes downward (or upward) the state process at the maximal rate whenever the state process is above (or below) a suitable threshold. As a result, the state process becomes the so-called a \emph{refracted process}, whose drift differs depending on whether it is above or below the threshold.

The study of the refracted process and the threshold strategy in stochastic control  
became active thanks to the seminal paper \cite{KypLoe2010}, where they developed the fluctuation theory of refracted \lev processes and derived various quantities of interest
 (e.g.\ first-passage times and resolvet measures) written in terms of the scale function. These  results provided a new and efficient approach in solving stochastic control problems under the absolutely-continuous condition. Their results were first applied in the absolutely-continuous version of the optimal dividend problem in \cite{KypLoePer2012}, followed by several papers on related stochastic control problems, e.g., \cite{CzaPerYam2018, HerPerYam2016,  PerYam2017_1}.
  All these papers focused on the spectrally negative or positive \lev case where the scale function is readily available to be used. 
Indeed, using the scale function, various quantities of interest related to first passage times and resolvent/occupation measures can be explicitly written; see \cite{boxma2017reinsurance, li2018weighted,  li2019number, 
 lkabous2017parisian, renaud2014time,wang2023refracted, yang2020parisian}.
There exist results on refracted processes driven by processes other than spectrally one-sided \lev processes such as \cite{cheung2016joint,renaud2021stochastic,zhang2014perturbed, zhou2017distribution}. In these papers, obtained identities are more involved, without the use of the scale function and do not consider optimimization problems.

However, existing results on stochastic control for a general \lev process with two-sided jumps are limited, especially, for the formulation under the absolutely-continuous condition. To the best of our knowledge, Noba \cite{NOBA2023174} is the only existing paper that successfully showed the optimality of refraction strategies for a general \lev case in a version of the problem they considered.
 The problem is more challenging than the spectrally one-sided case because the scale function approach is no longer available. On the other hand, the spectrally one-sided assumption -- although convenient for analytical tractability -- is often not suitable in real applications. For financial applications, asset prices are empirically known to jump both upward and downward \cite{carr2002fine}.  In inventory applications,  the dynamics of water levels in dams/weirs, for example, are suitably modeled in terms of processes with jumps in both directions due to  possible surges in consumption and rainfall.



In this paper, we revisit a classical problem in stochastic control, where the controller can push the state process in the negative direction. Given a state process following $X = (X_t)_{t \geq 0}$ in the absence of control,  a control process $L= (L_t)_{t \geq 0}$ models the cumulative reduction so that the controlled process becomes $X-L$. The objective is to derive an optimal control $L$ so that the expected value of the running cost $\int_0^\infty e^{-qt}f(X_t-L_t) \diff t$ and controlling cost $\int_0^\infty e^{-qt} \diff L_t$ is minimized for a discount factor $q > 0$. The absolutely-continuous version requires 
additionally that the control process admit the form $L_t = \int_0^t l_s \diff s$, $t \geq 0$, for a density $l$ taking values on $[0, \alpha]$ uniformly in time. We allow $X$ to be a \lev process with two-sided jumps.

This paper  generalizes \cite{HerPerYam2016} and studies a problem related to    \cite{noba2022stochastic} and \cite{NobYam2022}. 
  In \cite{HerPerYam2016}, they studied the case when $X$ is a spectrally negative \lev process and  solved the problem directly using the scale function, with which the candidate value function can be written explicitly.
 Again, this approach is not available when $X$ admits two-sided jumps and thus a completely different approach is required. Noba and Yamazaki \cite{NobYam2022} studied the case for a general \lev process $X$ for the classical model without the absolutely-continuous condition. Our paper is also related to a recent paper \cite{noba2022stochastic} for a general \lev process $X$ under a different condition where control opportunities arrive only at Poisson arrival times.  We are also motivated by \cite{NOBA2023174}, which showed the optimality of a refraction strategy in an optimal dividend problem,
  but our problem is distinct and more challenging as we deal with a running cost driven by a general convex cost function $f$.
 

As in the case of \cite{HerPerYam2016,  noba2022stochastic,NobYam2022}, we derive a concise expression of the optimal threshold $b^*$ (see Section \ref{subsection_candidate_barrier}) and verify rigorously the optimality of refracting the state process at the threshold. The key result for this is the probabilistic expression of the slope of the candidate value function as in \eqref{der_v}, with which the verification of optimality can be carried out in a straightforward fashion.

 Our approach to be used in this paper is motivated by the pathwise approach of  \cite{noba2022stochastic,NobYam2022} where they developed analogous expressions as  \eqref{der_v}.
  However, our problem is significantly more difficult than the cases of \cite{noba2022stochastic,NobYam2022}. This is due to the difference between the reflected process and the refracted process. When the reflected paths of $X$ when started at two different points are compared, these eventually coincide at a finite time and follow the same paths afterwards. On the other hand, refracted paths do not have this property. Due to this difference, the proof of the key result (Theorem \ref{density_v_b_star}) requires an original approach.




The rest of the paper is organized as follows: In Section \ref{SubSec301}, we formulate the considered problem, and, in Section \ref{SubSec202}, we review the refracted Levy process. Section \ref{section_main_results} presents the main optimality result. In Section \ref{sec_proof_density_v_b_star}, we provide the proof of Theorem \ref{density_v_b_star}.


\section{The problem}\label{SubSec301}

Let $X= (X_t)_{t \geq 0}$ be a one-dimensional L\'evy process, defined on a probability space $(\Omega , \cF , \bP)$, modeling the state dynamics in the absence of control. For all $x\in\bR$, we write $\bP_x$ for the law of $X$ when it starts at $x$ and write $\E_x$ for its expectation. In particular, when $x=0$, we drop the subscript.
Let $\Psi$ be the characteristic exponent of $X$, which satisfies
\begin{align}
e^{-t\Psi(\lambda)}=\bE \sbra{e^{i\lambda X_t}},  \quad \lambda \in \bR, \quad t\geq 0. 
\end{align}  
It is known that the characteristic exponent $\Psi$ has the form  
\begin{align}
\Psi (\lambda) = -i\gamma\lambda +\frac{1}{2}\sigma^2 \lambda^2 
+\int_{\bR\backslash\{0\}} (1-e^{i\lambda z}+i\lambda z1_{\{\absol{z}<1\}}) \Pi(\diff z) , \quad\lambda\in\bR, 
\end{align}
for some $\gamma\in\bR$, a diffusion coefficient $\sigma\geq 0$, and a L\'evy measure $\Pi $ on $\bR \backslash \{0\}$ satisfying 
\begin{align}
\int_{\bR\backslash \{ 0\}}(1\land z^2) \Pi (\diff z). 
\end{align}
In addition, {it is known} that the process $X$ has bounded variation paths if and only if $\sigma= 0$ and the L\'evy measure $\Pi$ satisfies $\int_{(-1, 1)\backslash\{0\}} \absol{z}\Pi(\diff z)<\infty$. 
When this holds, we denote
\begin{align}
\Psi(\lambda)= -i\delta\lambda+\int_{\bR\backslash\{0\}}(1-e^{i\lambda z})\Pi (\diff z),  \label{78}
\end{align}
where
\begin{align}
\delta := \gamma-\int_{(-1, 1)\backslash\{0\}}z \Pi(\diff z), \label{def_delta}
\end{align}
is the drift of $X$.
As is commonly assumed in the literature, we impose the following to guarantee that the value function is well-defined and finite.
\begin{Ass}\label{Ass301}
There exists $\bar{\theta}>0$ such that $\int_{\bR\backslash (-1, 1)} e^{\bar{\theta}|z|}\Pi (\diff z)<\infty$. 
\end{Ass}

We denote by $\bF=(\cF_t )_{t\geq 0}$ the natural filtration generated by $X$.  
A strategy is given by an $\mathbb{F}$-adapted and progressively measurable process $\pi=(l_t^\pi)_{t \geq 0}$, satisfying
\begin{align}
l_t^\pi \in [0, \alpha], \quad t\geq 0 \quad \textrm{a.s.} \label{l_bound}
\end{align} 
for a deterministic constant maximum rate $\alpha > 0$.
A strategy $\pi$ pushes downward the state process $X$: the resulting controlled process becomes
\begin{align}
U^\pi_t :=X_t -\int_0^t l^\pi_s \diff s, \quad t\geq 0. \label{6}
\end{align}
Let $\cA$ be the set of strategies satisfying the above conditions.

The goal is to derive an optimal strategy that minimizes the sum of the running and controlling costs. We fix a discount rate $q>0$ and a unit cost of controlling $\beta \in\bR$ (reward if it is negative). 
The net present value (NPV) of the expected total costs under the strategy 
$\pi$
is
\begin{align}
v_\pi (x) := \bE_x \sbra{\int_0^\infty e^{-qt} \rbra{f(U^\pi_t) + \beta l_t^\pi} \diff t}, \quad x \in \R,
\end{align}
where the cost function $f$ is assumed to satisfy the following throughout the paper.

\begin{Ass} \label{assump_f}
The running cost function $f:\bR \to \bR$ is given by a continuously-differentiable and convex function, which is of polynomial growth, i.e., 
\begin{align}
|f(x)| < k_1 + k_2 {|x|}^N, 
\end{align}
for some $k_1, k_2>0$ and $N\in\bN$. 
\end{Ass}

The problem is to compute
 the value function
\begin{align}
v (x)=\inf_{\pi \in\cA} v_\pi(x) ,\quad x\in\bR, 
\end{align} 
and find an optimal strategy $\pi^\ast \in \cA$ satisfying 
\begin{align}
v_{\pi^\ast} (x)=v (x),
\quad x\in\bR, \label{5}
\end{align} 
if such a strategy exists.


\begin{remark} \label{remark_X_alpha} 
If a strategy is to activate $l^\pi$ at  the maximal rate $\alpha$ at all times, then the controlled process follows a
 drift-changed \lev process 
$X^{(\alpha)}=(X^{(\alpha)}_t)_{t\geq 0}$ with 
\begin{align}\label{X_alpha}
X^{(\alpha)}_t := X_t -\alpha t, \quad t\geq 0. 
\end{align}
Then, for any admissible strategy $\pi \in \mathcal{A}$ (satisfying \eqref{l_bound}), we have $X^{(\alpha)}_t \leq U_t^\pi \leq  X_t$ uniformly in $t \geq 0$ a.s.
\end{remark}

Under Assumption \ref{Ass301}, we have the following lemma.  The proof is omitted because it follows from
the proof of \cite[Lemma 11]{Yam2017}.  
\begin{lemma} \label{Lem302}
For $x\in\bR$ and any function $g:\bR \to \bR$ of polynomial growth, 
\begin{align}
\bE_x \sbra{\int_0^\infty e^{-qt} |g(X_t)| \diff t}<\infty, \quad \text{and}\quad
\bE_x \sbra{\int_0^\infty e^{-qt} |g(X^{(\alpha)}_t)| \diff t}<\infty. \label{29}
\end{align} 
In addition, the maps $x\mapsto \bE_x \sbra{\int_0^\infty e^{-qt} |g(X_t)| \diff t}$ and $x\mapsto\bE_x \sbra{\int_0^\infty e^{-qt} |g(X^{(\alpha)}_t)| \diff t}$ are {of} polynomial growth.

{Thus, by Remark \ref{remark_X_alpha} and the convexity of $f$, 
$\bE_x \sbra{\int_0^\infty e^{-qt} |f(U^\pi_t) | \diff t  } < \infty$ for all $\pi \in \mathcal{A}$.}
\end{lemma}

 It is also clear that $\bE_x \sbra{\int_0^\infty e^{-qt}  l_t^\pi \diff t} \leq \alpha/q < \infty$ for all $x \in \R$.  Therefore, we have the following. 
\begin{lemma}
For all $\pi \in \mathcal{A}$, we have $-\infty < v_\pi (x) < \infty$ for $x \in \R$.
\end{lemma}

%

\section{Refracted L\'evy processes}\label{SubSec202}

In this paper, we prove that a refraction strategy with a suitable threshold is optimal.
In this section, we briefly review  the refraction strategy. Its definition in our setting requires technical details because we are considering a general \lev process whereas most existing results in the literature focus on the spectrally one-sided case (with an additional condition in the bounded variation case). Here, we follow the definition given in  \cite[Section 3.2]{NOBA2023174}, which depends on 
the following classifications.
%
\begin{description}
\item[Case $1$:] $X$ has unbounded variation paths or has bounded variation paths with $\delta \not \in [0, \alpha]$, 
\item[Case $2$:] $X$ has bounded variation paths with $\delta \in [0 , \alpha ]$ {(complement of Case $1$)}.
\end{description}
Here we recall that $\delta$ as in  \eqref{def_delta} is the drift of $X$ if it is of bounded variation whereas $\alpha$ is the upper bound of $l^\pi$ (see \eqref{l_bound}). These two cases need to be handled separately because, in Case 2, it is possible to control the drift to be nonpositive at all times (by choosing $l^\pi \geq \delta$) whereas, in Case 1, this is not possible.


Fix $b \in \bR$. We define the refraction strategy $\pi^b := (l^b_t)_{t \geq 0} \equiv (l^{\pi^b}_t)_{t \geq 0}$ with threshold $b$ to be such that it pushes downward the process at rate $\alpha$ whenever the process is above $b$, and at rate $\delta$ when the process is exactly at $b$ in Case 2. More precisely, when the controlled state process is at $y$, the rate $l^b$ is selected to be 
\begin{align}
h^b(y):=
\begin{cases}
\alpha1_{(b, \infty)}(y) \quad &\text{ in Case $1$}, \\
\alpha1_{(b, \infty)}(y)+\delta 1_{\{b\}}(y) \quad &\text{ in Case $2$}.
\end{cases}
\end{align}
The resulting controlled process $U^b$ can be written as a solution to  the stochastic differential equation (SDE)
\begin{align}
U_t^b = X_t - \int_0^t l^b_s \diff s, \quad t \geq 0,  \label{4}
\end{align}
where 
\begin{align}
l^b_t := h^b (U^b_t), \quad t\geq 0. 
\end{align}
%
%
%
%
It is clear that the refraction strategy  is  admissible. For the rest of the paper, we write its NPV function as $v_b(\cdot) := v_{\pi^b} (\cdot)$.

A majority of the existing results assume $\delta > \alpha$ for the case of bounded variation and thus focus on Case 1. In this case, on $(b, \infty)$, the process $U^b$ follows the refracted process of Kyprianou and Loeffen  \cite{KypLoe2010}. On the other hand, this paper considers also Case 2, where $b$ can be seen as a barrier (rather than a threshold). However, different from the classical reflected process with an upper barrier, the process $U^b$ can exceed $b$ if it jumps over $b$. 

For the rest of the paper, we {impose} the following assumption for $X$; this is not restrictive in view of Remark \ref{remark_assumption_SDE}. 
\begin{Ass}\label{Ass202a}
The SDE \eqref{4} has a unique strong solution $U^b$. 
\end{Ass}
\begin{remark} \label{remark_assumption_SDE}
From existing studies and \cite[Section 2.2]{NOBA2023174}, 
Assumption \ref{Ass202a} is guaranteed to hold if one or more of the following holds:
\begin{enumerate}
\item The process $X$ has bounded variation paths.
\item The process $X$ has unbounded variation paths and the L\'evy measure $\Pi$ satisfies $\Pi(-\infty, 0)<\infty$ or $\Pi(0, \infty)<\infty$. 
\item The process $X$ has a non-trivial Gaussian component (i.e.\ $\sigma > 0$).
\end{enumerate}
\end{remark}



%

%
%
%

\section{The main result} \label{section_main_results}

\subsection{Candidate barrier} \label{subsection_candidate_barrier}
Our candidate optimal threshold is succinctly defined by
\begin{align}\label{opt_thres}
b^\ast := \inf \{b\in \bR: \rho(b)-\beta \geq  0\},
\end{align}
where 
\begin{align}\label{fun_rho}
\rho(b):= \bE_b \sbra{\int_0^\infty e^{-qt} f^\prime (U^b_t) \diff t}=\bE \sbra{ \int_0^\infty e^{-qt}f^\prime (U^0_t+b ) \diff t }, \quad b \in \R. 
\end{align}
Here, it is understood that $b^*=-\infty$ if $\rho(b) - \beta \geq 0$ for all $b \in \mathbb{R}$ and  $b^*=+\infty$ if $\rho(b) - \beta < 0$ for all $b \in \mathbb{R}$.

Interestingly, the optimal barrier/threshold is written in the same way in the classical case \cite{NobYam2022} and also in a different version under the Poisson observation condition  \cite{noba2022stochastic}, with $U^b$ replaced by analogous reflected processes in their respective problems. In the absolutely-continuous model,  the same expression was obtained in the spectrally negative case in \cite{HerPerYam2016}.



We first  summarize, in the following remark, the properties of the function $\rho$, with which $b^*$ is defined. 
\begin{remark}\label{30}
\begin{enumerate}
\item By the convexity and the polynomial growth condition of $f$, its derivative $f^\prime$ is of polynomial growth because
\begin{align}
|f^\prime (x)| \leq |f(x+1)-f(x)| \lor |f(x)-f(x-1)|,\quad x\in\bR. \label{36}
\end{align}
Thus, the function $\rho$ is well-defined and finite by Lemma \ref{Lem302}. 

\item By the convexity of $f$, the function $\rho$ is non-decreasing.

\item By 
the dominated convergence theorem and Assumption \ref{assump_f}, the function $\rho$  is continuous on $\bR$. 

\item By monotone convergence, we have $\rho(\infty):=\lim_{b \to +\infty} \rho(b) = f'(\infty)/q \in (-\infty, \infty]$ and  $\rho(-\infty):=\lim_{b \to -\infty} \rho(b) = f'(-\infty)/q \in [-\infty, \infty)$ where $f'(\infty) := \lim_{b \to \infty} f'(b)$ and $f'(-\infty) := \lim_{b \to -\infty} f'(b)$ exist by the convexity of $f$.
\end{enumerate}
\end{remark}


By the above remark, the following is immediate.

\begin{lemma} \label{lemma_about_b_star}
\begin{enumerate}

\item {We have $b^*\in (-\infty, \infty)$ if and only if $f'(-\infty) < q\beta$ and $f'(x)\geq q\beta$ for some large $x$. Otherwise, $b^*=-\infty$ if and only if $f'(-\infty)\geq q\beta$, and $b^*=\infty$ if and only if $f'(x) <  q\beta$ for all $x \in \R$.}
\item When $b^* \in (-\infty, \infty)$, we have $\rho(b^\ast)=\beta$.
\end{enumerate}
\end{lemma}

\subsection{The smoothness and convexity of the function $v_{b^\ast}$}

We now provide a key theorem, with which the verification of optimality can be carried out in a straightforward way. The proof of this result is the most important contribution of this paper. Although similar results have also been observed in different versions of singular control problems  \cite{HerPerYam2016, noba2022stochastic,NobYam2022}, the absolutely-continuous version for a general \lev process is significantly more difficult. Because the proof is long and requires technical details of the path properties of  \lev processes, we defer the proof to Section \ref{sec_proof_density_v_b_star}.

\begin{theorem}\label{density_v_b_star}
The function $v_{b^\ast}$ 
{is convex and continuously differentiable and} 
admits a derivative written as
\begin{align}\label{der_v}
v_{b^\ast}^\prime (x) =\bE_x \sbra{\int_0^\infty e^{-qt} f^\prime(U^{b^\ast}_t)\diff t },\quad x\in\bR. 
\end{align}
In particular, when $b^*=+\infty$, the above holds with $U^{b^\ast} \equiv X$; when $b^*=-\infty$, the above holds with $U^{b^\ast} \equiv X^{(\alpha)}$.
\end{theorem}
%
%

The following is a corollary of Theorem \ref{density_v_b_star}.

\begin{corollary} \label{cor_property_v_b}
We have
\begin{equation}\label{der_HJB}
v'_{b^*}(x)
\begin{cases}
\geq \beta &  \text{if $x> b^*$,}\\
= \beta &  \text{if $x= b^*$,}\\
\leq \beta &  \text{if $x< b^*$.}
\end{cases}
\end{equation}

\end{corollary}
\begin{proof}
When $b^*\in(-\infty,\infty)$, the convexity along with $v'_{b^*}(b^*)=\beta$ as in Lemma \ref{lemma_about_b_star}(ii)
shows  \eqref{der_HJB}.

%
For the case $b^*=\infty$, we have by monotone convergence and Lemma  \ref{lemma_about_b_star}(i),
\[
v_{\infty}'(\infty)=\lim_{x\to\infty}\E\left[\int_0^\infty e^{-qt}f'(x+X_t)\diff t\right]=\frac{f'(\infty)}{q}\leq \beta. 
\]
Similarly, for $b^*=-\infty$ 
\[
v_{-\infty}'(-\infty)=\lim_{x\to-\infty}\E\left[\int_0^\infty e^{-qt}f'(x+X_t^{(\alpha)}) \diff t\right]=\frac{f'(-\infty)}{q}\geq \beta. 
\]
In either case, by Theorem \ref{density_v_b_star}, we have that $v_{b^*}'$ satisfies \eqref{der_HJB}. 
\end{proof}

\subsection{Verification lemma}\label{Subsec_Aux_Verification}
We first state a verification lemma, which provides a sufficient condition for a strategy to be optimal.

We define two classes of functions and an operator similar to those defined in \cite[Section 2.1]{Nob2021}. 

First, let $C^{(1)}_{\text{poly}}$ be the set of continuously-differentiable functions $g:\bR\to\bR$, which are of  polynomial growth, i.e.
there exist $k_1, k_2, M >0$ such that  
$
|g(x)|<k_1 + k_2 {|x|}^M
$
for all $x\in\bR$. 

%
Second, let $C^{(2)}_{\text{poly}}$ be a subset of $C^{(1)}_{\text{poly}}$ such that the continuous derivative  $g^\prime$  admits, on $\mathbb{R}$, a continuous density function $g^{\prime\prime}: \bR\to \bR$ which 
satisfies $g^\prime (x)+\int_x^y g^{\prime\prime}(z) \diff z = g^\prime(y)$ for $x, y\in\bR$ with $x<y$. 

Let $\mathcal{L}$ be the operator acting on functions in $C^{(1)}_{\text{poly}}$ (resp. $C^{(2)}_{\text{poly}}$) when $X$ has bounded (resp. unbounded) variation paths defined by
\begin{align}
\mathcal{L} g(x)= \gamma g^\prime(x) +\frac{1}{2}\sigma^2 g^{\prime\prime}(x) 
+\int_{\R \backslash \{0\} } (g(x+z)-g(x)-g^\prime(x) z1_{\{|z|<1\}}) \Pi (\diff z) , \quad x\in \mathbb{R}.\label{38}
\end{align}
Here for the case of unbounded variation, we fix the density $g''$ satisfying the requirements given in the definition of $C^{(2)}_{\text{poly}}$.
%
By Assumption \ref{Ass301} and following the same arguments as in \cite[Remark 2.4]{Nob2021}, the operator $\mathcal{L}$ is well-defined in $C^{(1)}_{\text{poly}}$ (resp. $C^{(2)}_{\text{poly}}$) for bounded (resp.\ unbounded) variation cases.


{
The verification lemma, given below, follows similar to \cite{HerPerYam2016, NOBA2023174} and thus we omit the proof.
\begin{lemma}\label{verification}
	Let $w: \R \to \R$ belong to $C^{(1)}_{\text{poly}}$ when $X$ is of bounded variation and $C^{(2)}_{\text{poly}}$ otherwise and satisfies
	\begin{align}\label{HJB_eq}
		&(\mathcal{L} -q )w(x)+ \inf_{0\leq r\leq \alpha}r\left(\beta-w'(x)\right)+f(x)\geq 0, \quad x\in\bR. 
	\end{align}
	Then we have $w(x) \leq \inf_{\pi\in \mathcal{A}}v_\pi(x)$ for $x \in \R$. 
\end{lemma}

\subsection{Optimality}

In order to apply Lemma \ref{verification}, we need to confirm that our candidate value function $v_{b^*}$ satisfies the smoothness condition: $v_{b^*} \in C^{(1)}_{\text{poly}}$ (resp. $v_{b^*} \in C^{(2)}_{\text{poly}}$) in the bounded (resp.\ unbounded) variation cases. The case of bounded variation is already taken care of by Theorem \ref{cor_property_v_b} and Lemma \ref{Lem302}
 (see the proof of Lemma \ref{prop_v_prime}(1) for a complete proof).  In addition, it is known, as in \cite{HerPerYam2016}, that it is satisfied when $X$ is a spectrally negative \lev process whether or not it is of bounded or unbounded variation.
However, as the class of \lev processes is wide, it is difficult to confirm that $v_{b^*} \in C^{(2)}_{\text{poly}}$ holds in general when $X$ is of unbounded variation. We therefore assume the following and then provide a very general sufficient condition.

\begin{Ass}\label{Ass401a}
When the process $X$ is of unbounded variation paths and it is not a spectrally negative \lev process, we assume $v_{b^*}'$ as in \eqref{der_v} has a continuous density.
\end{Ass}

However, this assumption is satisfied in essentially all practical applications in view of the following. Its proof is deferred to Appendix \ref{app_lemma}.

\begin{lemma}\label{Lem402}
If the L\'evy measure satisfies $\Pi(0,\infty)<\infty$ or $\Pi(-\infty,0)<\infty$, Assumption \ref{Ass401a} is satisfied.
\end{lemma}


We are now ready to state the optimality result. Below, we let $v_{b^*}''$ be the continuous density of $v_{b^*}'$ which exists by Assumption \ref{Ass401a}.

\begin{theorem}\label{Thm401}
	Under Assumptions  \ref{Ass301}, \ref{assump_f},  \ref{Ass202a},
	and \ref{Ass401a},
	the threshold strategy $\pi^{b^*}$ is optimal and the value function is given by $v(x)=v_{b^*}(x)=v_{\pi^*}(x)$ for all $x \in \R$.
\end{theorem}
To prove Theorem \ref{Thm401}, we provide the following two lemmas, which can be derived easily  using the smoothness results and assumptions. 
Below, we assume Assumptions  \ref{Ass301}, \ref{assump_f},  \ref{Ass202a},
	and \ref{Ass401a}.

\begin{lemma}\label{prop_v_prime}
(1) The function 
$v_{b^*}$
belongs to $C^{(1)}_{\text{poly}}$.
(2) When $X$ is of unbounded variation paths, $v_{b^*}$ belongs to  $C^{(2)}_{\text{poly}}$.
\end{lemma}
\begin{proof}
(1) It is immediate by Lemma \ref{Lem302} and Theorem \ref{density_v_b_star}.
%
(2) When the process $X$ has unbounded variation paths, by Assumption \ref{Ass401a},
we have that $v_{b^*}$ belongs to  $C^{(2)}_{\text{poly}}$.
\end{proof}

The following lemma can be shown via a standard martingale arguments as those in
 \cite[Lemma 5.7]{Nob2021} (see also \cite[Lemma 5.5]{NOBA2023174} and \cite[Lemma 4.1]{HerPerYam2016}),
  and thus we omit the proof.

 

\begin{lemma} \label{lemma_martingale} 
We have \begin{align}
\left(\mathcal{L}-q\right)v_{b^*}(x)+f(x)&=0,\qquad x<b^*, \\ 
\left(\mathcal{L}-q\right)v_{b^*}(x)+\alpha(\beta-v_{b^*}'(x))+f(x)&=0,\qquad x \geq b^\ast. 
\end{align}
\end{lemma}

We are now ready to complete the proof of Theorem \ref{Thm401}.

\begin{proof}[Proof of Theorem \ref{Thm401}]
	By Lemma \ref{prop_v_prime}, the function $v_{b^\ast}$ is sufficiently smooth and is of polynomial growth and thus we can apply the operator $\mathcal{L}$. 
	In addition, 
	combining Corollary \ref{cor_property_v_b} and Lemma \ref{lemma_martingale},
	the function $v_{b^\ast}$ satisfies \eqref{HJB_eq}. 
	Therefore, by Lemma \ref{verification}, $v_{b^*}(x) \leq \inf_{\pi\in \mathcal{A}}v_\pi(x)$ for $x \in \R$. This also holds with equality because $\pi^{b^\ast}$ is admissible. This completes the proof. 
\end{proof}

\section{Proof of Theorem \ref{density_v_b_star}} \label{sec_proof_density_v_b_star}


We conclude this paper with a detailed
 proof of Theorem \ref{density_v_b_star}, which provided the key result essential for the verification of optimality in the previous section. We first consider the case $b^* \in (-\infty,\infty)$ and then the case $b^* = \pm \infty$. 

\subsection{For the case $b^* \in \mathbb{R}$}



For $\varepsilon \in \mathbb{R}$, let $X^{[\varepsilon]}=\{X^{[\varepsilon]}_t {=X_t +\varepsilon} : t\geq 0\}$ be the shifted process of $X$ by $\varepsilon$.
In addition, let $l^{[\varepsilon],b^\ast}=\{l^{[\varepsilon], b^\ast}_t : t\geq 0\}$ be the refraction strategy at level $b^*$ and $U^{[\varepsilon], b^\ast}=\{U^{[\varepsilon], b^\ast}_t : t\geq 0\}$ its corresponding refracted L\'evy process driven by $X^{[\varepsilon]}$; i.e. $U^{[\varepsilon], b^\ast}_t=X^{[\varepsilon]}_t - L^{[\varepsilon], b^\ast}_t$ 
where $ L^{[\varepsilon], b^\ast}_t:= \int_0^t  l^{[\varepsilon], b^\ast}_s \diff s $ for $t\geq 0$.  It is easy to confirm that
\begin{align} 
l^{[\varepsilon], b^\ast}_t = l^{b^\ast-\varepsilon}_t \quad {\textrm{and} \quad U^{[\varepsilon], b^\ast}_t = U^{b^\ast-\varepsilon}_t + \varepsilon, } \quad t\geq 0.
\end{align}


Then, by the same argument as that of
 the proof of \cite[Lemma D.1]{NOBA2023174}, 
{the following properties hold for the difference of the control and controlled processes.}
\begin{lemma} \label{lemma_property_diff_process} 
For 
$\varepsilon > 0$, the following holds a.s.:
\begin{enumerate}
\item[(a)] The map $t \mapsto U^{[\varepsilon], b^\ast}_t-U^{b^\ast}_t$
 is non-increasing and takes values on $[0, \varepsilon]$; 
 \item[(b)] The map $t \mapsto L^{[\varepsilon], b^\ast}_t-L^{b^\ast}_t$
 is non-decreasing and takes values on $[0, \varepsilon]$; 
 \item[(c)] The support of the Stieltjes measure of $\{U^{[\varepsilon], b^\ast}_t-U^{b^\ast}_t: t \geq 0\}$ and $\{L^{[\varepsilon], b^\ast}_t - L^{b^\ast}_t; t  \geq  0\}$ are in the closure of $\{ t\geq 0 :  b^* \in  [ U^{b^\ast }_t, U^{[\varepsilon],b^\ast}_t  ], 
U^{b^\ast}_t \neq U^{[\varepsilon],b^\ast }_t\}$.
\end{enumerate}
\end{lemma}

In Lemma \ref{lemma_property_diff_process}, the property (c) means that the difference of the controlled process changes only when $U^{[\varepsilon],b^\ast}$ is in the control region and when $U^{b^\ast}$ is in the uncontrol region.


Fix  $\varepsilon> 0$ and  $x\in\bR$. We have 
\begin{align}
v_{b^\ast}(x+\varepsilon)-v_{b^\ast}(x)=
u_1 (x, \varepsilon) + \beta u_2 (x, \varepsilon), \label{7}
\end{align}
where
\begin{align}
u_1 (x, \varepsilon) &:= \bE_x \sbra{\int_0^\infty e^{-qt}\rbra{f(U^{[\varepsilon], b^\ast}_t)-f(U^{  b^\ast}_t) }\diff t}, \\
u_2 (x, \varepsilon) &:=  \bE_x \sbra{\int_0^\infty e^{-qt} \rbra{l^{[\varepsilon],b^\ast}_t-l^{  b^\ast}_t} \diff t}.
\end{align}


We want to give bounds on \eqref{7} to compute the right derivative of $v_{b^\ast}$. 
To this end, we take a sequence of stopping times 
\[
0 =: T^{(0)}_{[\eta]} \leq S^{(1)}_{[\eta]} < T^{(1)}_{[\eta]} \leq S^{(2)}_{[\eta]} < \cdots
\]
so that $t \mapsto L^{[\varepsilon], b^\ast}_{t}-L^{  b^\ast}_{t }$ may increase (equivalently, $t \mapsto U^{[\varepsilon], b^\ast}_{t}-U^{  b^\ast}_{t }$ may decrease) only at $t \in \bigcup_{n\in\bN} [{S^{(n)}_{[\eta]}, T^{(n)}_{[\eta]}}]$ and stay constant on its complement  $\bigcup_{n\in\bN} (T^{(n-1)}_{[\eta]}, S^{(n)}_{[\eta]} )$. This is done by recursively defining by, for $n \geq 1$,
\begin{align}
S^{(n)}_{[\eta]} &:= 
\inf \cbra{ t>T^{(n-1)}_{[\eta]}:  U^{  b^\ast}_t\leq b^\ast, U^{[\varepsilon], b^\ast}_t \geq b^\ast, U^{  b^\ast}_t \neq  U^{[\varepsilon], b^\ast}_t  }, \\
T^{(n)}_{[\eta]} &:= 
S^{(n)}_{[\eta]}+\eta,  
\end{align}
for a constant $\eta > 0$.
It is noted that $S^{(n)}_{[\eta]}$ and $T^{(n-1)}_{[\eta]}$ can coincide. But it is clear that both $S^{(n)}_{[\eta]}$ and $T^{(n)}_{[\eta]}$ go to infinity as $n \to \infty$ (and can be infinite for finite $n$). 
By Lemma \ref{lemma_property_diff_process}(c), it is indeed that $t \mapsto L^{[\varepsilon], b^\ast}_{t}-L^{  b^\ast}_{t }$ may increase only on $\bigcup_{n\in\bN} [{S^{(n)}_{[\eta]}, T^{(n)}_{[\eta]}}]$. 

We define, for $n \in \mathbb{N}$,
\begin{align}
\varepsilon^{(n)}_{[\eta]}{:=}&
\rbra{L^{[\varepsilon], b^\ast}_{T^{(n)}_{[\eta]}}-L^{  b^\ast}_{T^{(n)}_{[\eta]}}}-\rbra{L^{[\varepsilon], b^\ast}_{S^{(n)}_{[\eta]}}-L^{  b^\ast}_{S^{(n)}_{[\eta]}}}=\int_{S^{(n)}_{[\eta]}}^{T^{(n)}_{[\eta]}}\rbra{l^{[\varepsilon], b^\ast}_t-l^{ b^\ast}_t}\diff t, 
\label{epsilon_def}
\end{align} 
to be the 
amount of change during $[S^{(n)}_\eta, T^{(n)}_\eta]$.
From Lemma \ref{lemma_property_diff_process}(b), we have 
\begin{align}
\varepsilon^{(n)}_{[\eta]}\geq 0, \quad n\in\bN \label{Rev002}
\end{align} 
and
\begin{align}
\sum_{n\in{\bN}}\varepsilon^{(n)}_{[\eta]} {=\lim_{t \to \infty} (L^{[\varepsilon], b^\ast}_t-L^{b^\ast}_t}) \leq \varepsilon. \label{epsilon_bound}
\end{align} 

By these properties, 
we can rewrite 
\begin{align} \label{u_2_decomposition}
\begin{aligned}
u_2(x, \varepsilon) = \bE_x \sbra{\sum_{n\in\bN} \int_{S^{(n)}_{[\eta]}}^{T^{(n)}_{[\eta]}} e^{-qt} 
(l^{[\varepsilon],b^\ast}_t-l^{  b^\ast}_t ) \diff t},
\end{aligned}
\end{align}
where it is understood that $\int_\infty^\infty (\cdot) \diff t =0$.  

\begin{lemma} We have
\begin{align} \label{beta_epsilon_bound}
\left| \beta u_2(x, \varepsilon)
 - \beta \bE_x \sbra{\sum_{n\in\bN} e^{-qT^{(n)}_{[\eta]}} \varepsilon^{(n)}_{[\eta]} } \right| \leq |\beta| \varepsilon (1 - e^{-q \eta}).
\end{align}
\end{lemma}
\begin{proof}
 {From \eqref{Rev002} and the definition of $\varepsilon^{(n)}_{[\eta]}$ {as in \eqref{epsilon_def}}, 
 we have}
\begin{multline}
\bE_x \sbra{\sum_{n\in\bN} e^{-qT^{(n)}_{[\eta]}} \varepsilon^{(n)}_{[\eta]} }
\leq
u_2(x, \varepsilon) 
\leq 
\bE_x \sbra{\sum_{n\in\bN} e^{-qS^{(n)}_{[\eta]}}\varepsilon^{(n)}_{[\eta]}} 
 \leq
\bE_x \sbra{\sum_{n\in\bN} e^{-qT^{(n)}_{[\eta]}}\varepsilon^{(n)}_{[\eta]}}
+ \varepsilon (1 -e^{-q\eta}),
\label{a009}
 \end{multline}
 {
where the last inequality holds because \eqref{epsilon_bound} gives
\begin{multline} 
\bE_x \sbra{\sum_{n\in\bN} (e^{-qS^{(n)}_{[\eta]}} - e^{-qT^{(n)}_{[\eta]}}) \varepsilon^{(n)}_{[\eta]}} = 
\bE_x \sbra{\sum_{n\in\bN} (e^{-qS^{(n)}_{[\eta]}} - e^{-q (S^{(n)}_{[\eta]} + \eta)}) \varepsilon^{(n)}_{[\eta]}} \\
= (1- e^{-q \eta}) \bE_x \sbra{\sum_{n\in\bN} e^{-qS^{(n)}_{[\eta]}} \varepsilon^{(n)}_{[\eta]}} \leq  (1- e^{-q \eta}) \bE_x \sbra{\sum_{n\in\bN} \varepsilon^{(n)}_{[\eta]}} \leq \varepsilon (1 - e^{-q \eta}).
\end{multline}
}
By \eqref{u_2_decomposition} and \eqref{a009}, the proof is complete.
\end{proof}

We have 
\begin{align}
\int_0^\infty e^{-qt}
\frac{f(U^{[\varepsilon], b^\ast}_t)-f(U^{  b^\ast}_t) }{\varepsilon}\diff t
&=
F_1(\varepsilon)-F_2(\varepsilon),
\label{9}
\end{align}
where 
\begin{align}
F_1(\varepsilon) &:= \int_0^\infty
e^{-qt}
\frac{f(U^{[\varepsilon], b^\ast}_t)-f(U^{  b^\ast}_t) }{U^{[\varepsilon], b^\ast}_t-U^{  b^\ast}_t }\diff t, \\
F_2(\varepsilon) &:=\int_0^\infty
\frac{\varepsilon -\rbra{U^{[\varepsilon], b^\ast}_t-U^{  b^\ast}_t}  }{\varepsilon}
e^{-qt}
\frac{f(U^{[\varepsilon], b^\ast}_t)-f(U^{  b^\ast}_t) }{U^{[\varepsilon], b^\ast}_t-U^{  b^\ast}_t }\diff t.
\end{align}
Since $f$ is {convex}, 
\begin{align}
{ f^\prime(U^{b^\ast}_t )} \leq
\frac{f(U^{[\varepsilon], b^\ast}_t)-f(U^{  b^\ast}_t) }{U^{[\varepsilon], b^\ast}_t-U^{  b^\ast}_t }
\leq {f^\prime(U^{[\varepsilon], b^\ast}_t ) }
,\quad t\geq 0, \label{10}
\end{align}
{and thus
\begin{align} \label{F_1_bound}
\int_0^\infty e^{-qt}f^\prime(U^{b^\ast}_t )\diff t \leq F_1(\varepsilon) \leq \int_0^\infty e^{-qt}f^\prime(U^{[\varepsilon],b^\ast}_t )\diff t.
\end{align}
}

Since $U^{[\varepsilon], b^\ast}_t+L^{[\varepsilon], b^\ast}_t =X_t + \varepsilon
=U^{  b^\ast}_t+L^{  b^\ast}_t + \varepsilon $ for $t\geq 0 $, we have, for $n\in\bN$,   
\begin{align}
\frac{\varepsilon -\rbra{U^{[\varepsilon], b^\ast}_t-U^{  b^\ast}_t}  }{\varepsilon}
=\frac{L^{[\varepsilon], b^\ast}_t-L^{  b^\ast}_t  }{\varepsilon}
=\left\{ 
\begin{array}{ll}
\sum_{k=1}^{n-1} \frac{\varepsilon^{(k)}_{[\eta]} }{\varepsilon}, 
& t\in[T^{(n-1)}_{[\eta]}, S^{(n)}_{[\eta]} ), \\
\sum_{k=1}^{n-1}\frac{\varepsilon^{(k)}_{[\eta]} }{\varepsilon}
+a^{(n)}_t, & t\in[S^{(n)}_{[\eta]}, T^{(n)}_{[\eta]} ), \end{array} \right.
\label{11}
\end{align}
for some non-decreasing process $a^{(n)}=\{a^{(n)}_t: t\in [S^{(n)}_{[\eta]}, T^{(n)}_{[\eta]} )\}$  satisfying 
\begin{align}a^{(n)}_t \in \left[0, \varepsilon^{(n)}_{[\eta]}/\varepsilon \right]. \label{a_bound}
\end{align} 

On the other hand, we have
\begin{align}
{F_2(\varepsilon)}
= \sum_{n \in \mathbb{N}} (A_n + B_n),
\end{align}
where, by \eqref{11}, 
\begin{align}
A_n &:= \int_{T^{(n-1)}_{[\eta]}}^{S^{(n)}_{[\eta]}}
\frac{\varepsilon -\rbra{U^{[\varepsilon], b^\ast}_t-U^{  b^\ast}_t}  }{\varepsilon}
e^{-qt}
\frac{f(U^{[\varepsilon], b^\ast}_t)-f(U^{  b^\ast}_t) }{U^{[\varepsilon], b^\ast}_t-U^{  b^\ast}_t }\diff t \\
&=  \sum_{k=1}^{n-1} \frac{\varepsilon^{(k)}_{[\eta]} }{\varepsilon}\int_{T^{(n-1)}_{[\eta]}}^{S^{(n)}_{[\eta]}}
e^{-qt}
\frac{f(U^{[\varepsilon], b^\ast}_t)-f(U^{  b^\ast}_t) }{U^{[\varepsilon], b^\ast}_t-U^{  b^\ast}_t }\diff t, \\
B_n &:= \int_{S^{(n)}_{[\eta]}}^{T^{(n)}_{[\eta]}}
\frac{\varepsilon -\rbra{U^{[\varepsilon], b^\ast}_t-U^{  b^\ast}_t}  }{\varepsilon}
e^{-qt}
\frac{f(U^{[\varepsilon], b^\ast}_t)-f(U^{  b^\ast}_t) }{U^{[\varepsilon], b^\ast}_t-U^{  b^\ast}_t }\diff t \\
&= \int_{S^{(n)}_{[\eta]}}^{T^{(n)}_{[\eta]}}\left(\sum_{k=1}^{n-1}\frac{\varepsilon^{(k)}_{[\eta]} }{\varepsilon}
+a^{(n)}_t \right) e^{-qt}
\frac{f(U^{[\varepsilon], b^\ast}_t)-f(U^{  b^\ast}_t) }{U^{[\varepsilon], b^\ast}_t-U^{  b^\ast}_t }\diff t.
\end{align}

From \eqref{10}, we have 
\begin{align}
 \sum_{k=1}^{n-1} \frac{\varepsilon^{(k)}_{[\eta]} }{\varepsilon}
 \int_{T^{(n-1)}_{[\eta]}}^{S^{(n)}_{[\eta]}}
e^{-qt}f^\prime (U^{b^\ast}_t)\diff t\leq
A_n
&\leq \sum_{k=1}^{n-1} \frac{\varepsilon^{(k)}_{[\eta]} }{\varepsilon}
 \int_{T^{(n-1)}_{[\eta]}}^{S^{(n)}_{[\eta]}}
e^{-qt}f^\prime (U^{[\varepsilon], b^\ast}_t)\diff t, \label{16}
\end{align}
and thus by Fubini's theorem,
\begin{multline}
 \sum_{k\in\bN} \frac{\varepsilon^{(k)}_{[\eta]} }{\varepsilon}
 \sum_{n=k+1}^\infty
 \int_{T^{(n-1)}_{[\eta]}}^{S^{(n)}_{[\eta]}}
e^{-qt}f^\prime (U^{b^\ast}_t)\diff t
\leq\sum_{n\in\bN} A_n 
  \leq
 \sum_{k\in\bN} \frac{\varepsilon^{(k)}_{[\eta]} }{\varepsilon}
 \sum_{n=k+1}^\infty
 \int_{T^{(n-1)}_{[\eta]}}^{S^{(n)}_{[\eta]}}
e^{-qt}f^\prime (U^{[\varepsilon], b^\ast}_t)\diff t.
\label{26}
\end{multline}


From \eqref{10} and \eqref{11}, together with \eqref{a_bound},
\begin{multline}
\sum_{k=1}^{n-1}\frac{\varepsilon^{(k)}_{[\eta]} }{\varepsilon}\int_{S^{(n)}_{[\eta]}}^{T^{(n)}_{[\eta]}} e^{-qt}f^\prime(U^{ b^\ast}_t )\diff t
- \frac{\varepsilon^{(n)}_{[\eta]} }{\varepsilon}\int_{S^{(n)}_{[\eta]}}^{T^{(n)}_{[\eta]}} e^{-qt} \left[ f^\prime(U^{ b^\ast}_t ) \right]^- \diff t\\
\leq
B_n
\leq
\sum_{k=1}^{n-1}\frac{\varepsilon^{(k)}_{[\eta]} }{\varepsilon}\int_{S^{(n)}_{[\eta]}}^{T^{(n)}_{[\eta]}} e^{-qt}f^\prime(U^{[\varepsilon], b^\ast}_t )\diff t+
\frac{\varepsilon^{(n)}_{[\eta]} }{\varepsilon}\int_{S^{(n)}_{[\eta]}}^{T^{(n)}_{[\eta]}} e^{-qt} \left[f^\prime(U^{[\varepsilon], b^\ast}_t ) \right]^+ \diff t,
\label{17}
\end{multline}
where we denote $[x]^+ := x \vee 0$ and $[x]^- := -(x \wedge 0)$ so that $x = [x]^+ - [x]^-$.

Thus, by summing up and Fubini's theorem,
\begin{align}
\begin{aligned}
&\sum_{k\in\bN}  \frac{\varepsilon^{(k)}_{[\eta]} }{\varepsilon}
\rbra{ - \int_{S^{(k)}_{[\eta]}}^{T^{(k)}_{[\eta]}} e^{-qt}\left[ f^\prime(U^{ b^\ast}_t ) \right]^- \diff t
+\sum_{n=k+1}^\infty \int_{S^{(n)}_{[\eta]}}^{T^{(n)}_{[\eta]}} e^{-qt}f^\prime(U^{ b^\ast}_t )\diff t}\leq
\sum_{n\in\bN}B_n
\\
&
\leq
\sum_{k\in\bN}  \frac{\varepsilon^{(k)}_{[\eta]} }{\varepsilon}
\rbra{\int_{S^{(k)}_{[\eta]}}^{T^{(k)}_{[\eta]}} e^{-qt}\left[ f^\prime(U^{[\varepsilon], b^\ast}_t ) \right]^+ \diff t
+\sum_{n=k+1}^\infty \int_{S^{(n)}_{[\eta]}}^{T^{(n)}_{[\eta]}} e^{-qt}f^\prime(U^{ [\varepsilon]. b^\ast}_t )\diff t}.
\end{aligned}
\label{27}
\end{align}
Thus, noting $\sum_{n=k+1}^\infty
 \big(\int_{T^{(n-1)}_{[\eta]}}^{S^{(n)}_{[\eta]}}
e^{-qt}f^\prime (U^{b^\ast}_t)\diff t +   \int_{S^{(n)}_{[\eta]}}^{T^{(n)}_{[\eta]}} e^{-qt}f^\prime(U^{ b^\ast}_t )\diff t  \big) = \int_{T^{(k)}_{[\eta]}}^\infty e^{-qt}f^\prime(U^{b^\ast}_t )\diff t$  and similar for the case when $U^{b^*}$ is replaced with $U^{[\varepsilon], b^*}$, 
%
from \eqref{9}, \eqref{F_1_bound}, \eqref{26} and \eqref{27}, we have 
\begin{align}
\begin{aligned}
&\int_0^\infty e^{-qt}f^\prime(U^{ b^\ast}_t )\diff t
-\sum_{k\in\bN}\frac{\varepsilon^{(k)}_{[\eta]} }{\varepsilon}\rbra{\int_{S^{(k)}_{[\eta]}}^{T^{(k)}_{[\eta]}} e^{-qt}\left[ f^\prime(U^{[\varepsilon],b^\ast}_t ) \right]^+ \diff t
+\int_{T^{(k)}_{[\eta]}}^\infty e^{-qt}f^\prime(U^{[\varepsilon],b^\ast}_t )\diff t }
\\
&\leq\int_0^\infty e^{-qt}\frac{f(U^{[\varepsilon], b^\ast}_t)-f(U^{  b^\ast}_t) }{\varepsilon}\diff t
\\
&\leq \int_0^\infty e^{-qt}f^\prime(U^{[\varepsilon],b^\ast}_t )\diff t
-\sum_{k\in\bN}\frac{\varepsilon^{(k)}_{[\eta]} }{\varepsilon}\rbra{- \int_{S^{(k)}_{[\eta]}}^{T^{(k)}_{[\eta]}} e^{-qt}\left[ f^\prime(U^{ b^\ast}_t ) \right]^- \diff t
+\int_{T^{(k)}_{[\eta]}}^\infty e^{-qt}f^\prime(U^{b^\ast}_t )\diff t  }
.
\end{aligned}
\label{12}
\end{align}
Using the strong Markov property at $T^{(k)}_{[\eta]}$ {and because $\varepsilon^{(k)}_{[\eta]}$ is $\mathcal{F}_{T^{(k)}_{[\eta]}}$-measurable}, we have 
\begin{align}
\bE_x\sbra{ \frac{\varepsilon^{(k)}_{[\eta]} }{\varepsilon}\rbra{\int_{T^{(k)}_{[\eta]}}^\infty e^{-qt}f^\prime(U^{[\varepsilon], b^\ast}_t )\diff t}} &=\bE_x\sbra{ \frac{\varepsilon^{(k)}_{[\eta]} }{\varepsilon}\rbra{e^{-qT^{(k)}_{[\eta]}} R^{(q)}(U^{[\varepsilon], b^\ast}_{T^{(k)}_{[\eta]}})}  } \label{a012}
\end{align}
and
\begin{align}
\bE_x\sbra{ \frac{\varepsilon^{(k)}_{[\eta]} }{\varepsilon}\rbra{\int_{T^{(k)}_{[\eta]}}^\infty e^{-qt}f^\prime(U^{b^\ast}_t )\diff t}} =\bE_x\sbra{ \frac{\varepsilon^{(k)}_{[\eta]} }{\varepsilon}\rbra{e^{-qT^{(k)}_{[\eta]}} R^{(q)}(U^{b^\ast}_{T^{(k)}_{[\eta]}})  }},  \label{14}
\end{align}
where $R^{(q)}(x):=\E_x\left[\int_0^{\infty}e^{-qt} f'(U_t^{b^*}) \diff t\right]$. 
From 
\eqref{7}, \eqref{beta_epsilon_bound}, \eqref{12}, \eqref{a012} and \eqref{14},
we have 
\begin{multline}
\bE_x \sbra{\int_0^\infty e^{-qt}f^\prime(U^{b^\ast}_t)\diff t}- |\beta|(1-e^{-q\eta}) {- \delta_1(\eta)}\\
\leq\frac{v_{b^\ast}(x+\varepsilon)-v_{b^\ast}(x)}{\varepsilon} \leq \bE_x \sbra{\int_0^\infty e^{-qt}f^\prime(U^{[\varepsilon], b^\ast}_t )\diff t}+|\beta|(1-e^{-q\eta}) - \delta_2(\eta),
\label{35}
\end{multline}
where
\begin{align}
\delta_1(\eta) &:= \sum_{k\in\bN}\bE_x\sbra{ \frac{\varepsilon^{(k)}_{[\eta]} }{\varepsilon}\rbra{\int_{S^{(k)}_{[\eta]}}^{T^{(k)}_{[\eta]}} e^{-qt} \left[ f^\prime(U^{[\varepsilon], b^\ast}_t) \right]^+ \diff t
+e^{-qT^{(k)}_{[\eta]}}\rbra{ R^{(q)}(U^{[\varepsilon], b^\ast}_{T^{(k)}_{[\eta]}}) -\beta} }}, \\
\delta_2(\eta) &:=\sum_{k\in\bN}\bE_x\sbra{ \frac{\varepsilon^{(k)}_{[\eta]} }{\varepsilon}\rbra{-\int_{S^{(k)}_{[\eta]}}^{T^{(k)}_{[\eta]}} e^{-qt} \left[ f^\prime(U^{b^\ast}_t) \right]^- \diff t
+e^{-qT^{(k)}_{[\eta]}}\rbra{ R^{(q)}(U^{b^\ast}_{T^{(k)}_{[\eta]}}) -\beta} }}.
\end{align}


We now take $\eta \downarrow 0$ in \eqref{35}. To this end, we show the following two lemmas.

\begin{lemma} \label{lemma_delta_2_limit}
We have that
\begin{enumerate}
\item $\varliminf_{\eta \downarrow 0} \delta_{2}(\eta)
\geq 0$;
\item $\varlimsup_{\eta \downarrow 0} \delta_{1}(\eta)\leq 0$.
\end{enumerate}
\end{lemma}
\begin{proof}
(i) Because $S^{(k)}_{[\eta]} +\eta = T^{(k)}_{[\eta]}$ and $R^{(q)}$ is non-decreasing,
\begin{align}
\delta_{2}(\eta) 
&=\sum_{k\in\bN}\bE_x\sbra{ \frac{\varepsilon^{(k)}_{[\eta]} }{\varepsilon}
\int_{S^{(k)}_{[\eta]}}^{T^{(k)}_{[\eta]}} \rbra{ - e^{-qt}\left[ f^\prime(U^{b^\ast}_t)\right]^- 
+ \frac{1}{\eta}
e^{-qT^{(k)}_{[\eta]}}\rbra{ R^{(q)}(U^{b^\ast}_{T^{(k)}_{[\eta]}}) -\beta} } \diff t }\\
&\geq - \bE_x\sbra{
\int_0^\infty  e^{-qt} G_t(\eta) \diff t }
, \label{18}
\end{align}
where
\begin{align}
G_t(\eta) &:=  \sum_{k\in\bN} 1_{(S^{(k)}_{[\eta]}, T^{(k)}_{[\eta]} )} {(t)}\rbra{\frac{\varepsilon^{(k)}_{[\eta]} }{\varepsilon} \left[ f^\prime(U^{b^\ast}_t) \right]^- + \frac{1}{\varepsilon}\frac{\varepsilon^{(k)}_{[\eta]} }{\eta} \left[ R^{(q)} \left( 
\inf_{s \in [t, t+\eta]}U^{b^\ast}_s \right) -\beta \right]^- } \\ &\geq 0.
\end{align}
In view of \eqref{epsilon_def} and the bound \eqref{l_bound},
$\varepsilon^{(k)}_{[\eta]} / \eta \leq \alpha$ for $k\in\bN$. In addition, by the convexity of $f$,  $R^{(q)}$ and $f'$ are non-decreasing and $X^{(\alpha)} \leq U^{b^*}$. Thus, $R^{(q)} \left( 
\inf_{s \in [t, t+\eta]}U^{b^\ast}_s \right) \geq R^{(q)} ( \inf_{s \in[0, t+\eta] }X^{(\alpha)}_s )$ and $f^\prime(U^{b^\ast}_t) \geq f^\prime (X^{(\alpha)}_t)$. 
Hence, using \eqref{epsilon_bound}, 
%
%
%
\begin{align}
G_t(\eta) \leq  [f^\prime (X^{(\alpha)}_t) ]^-+ \frac{\alpha}{\varepsilon} \left[R^{(q)} \left( \inf_{s \in[0, t+\eta] }X^{(\alpha)}_s \right) -\beta \right]^-. \label{19}
\end{align}
By Lemma \ref{Lem302} and since $f^\prime$ {has} polynomial growth,
we have
\begin{align}
\bE_x\sbra{\int_0^\infty e^{-qt} \left[ f^\prime (X^{(\alpha)}_t) \right]^- \diff t }< \infty. \label{20}
\end{align}
In addition, the function $R^{(q)}$ is of polynomial growth by the same argument as that of the proof of \cite[Lemma 3]{NobYam2022}. 
Thus, by \cite[Exercise 7.1(ii)]{Kyp2014},  we have, {with ${\bf{e}}_q$ an independent exponential random variable with parameter $q$,} with $\underline{X}^{(\alpha)}_t := \inf_{0 \leq s \leq t}X^{(\alpha)}_s$,
\begin{multline}
\begin{split}
&\bE_x\sbra{\int_0^\infty e^{-qt}  \left[R^{(q)} \left( \underline{X}^{(\alpha)}_{t+\eta} \right) - \beta \right]^- \diff t } \\
&=\frac{1}{q}\bE_x\sbra{ \left[R^{(q)} \left( \underline{X}^{(\alpha)}_{{\bf{e}}_q+\eta} \right) - \beta \right]^- }=\frac{1}{q}\bE_x\sbra{ \left[R^{(q)} \left( \underline{X}^{(\alpha)}_{{\bf{e}}_q} \right) - \beta \right]^-\Big{|}{\bf{e}}_q>\eta}  \\
&= \frac{e^{q \eta}}{q}\bE_x\sbra{ \left[R^{(q)} \left( \underline{X}^{(\alpha)}_{{\bf{e}}_q} \right) - \beta \right]^-; {\bf{e}}_q>\eta} \leq \frac{e^{q\eta}}{q}\bE_x\sbra{ \left[R^{(q)} \left(  \underline{X}^{(\alpha)}_{{\bf{e}}_q} \right) -\beta \right]^- } \\ &<\infty.
\end{split}
\label{21}
\end{multline}
By \eqref{19}, \eqref{20} and \eqref{21}, we can use reverse Fatou's lemma and have
\begin{multline}
\varlimsup_{\eta \downarrow 0}
\bE_x\sbra{
\int_0^\infty  e^{-qt} {G_t(\eta)}
\diff t } \leq
\bE_x\sbra{
\int_0^\infty  e^{-qt} \varlimsup_{\eta \downarrow 0} {G_t(\eta)}
\diff t }  \leq
\bE_x\sbra{
\int_0^\infty  e^{-qt} \varlimsup_{\eta \downarrow 0}
(Z_t^{(1)} (\eta) + Z_t^{(2)} (\eta))
\diff t },\label{22}
\end{multline}
where we define
\begin{align}
Z_t^{(1)} (\eta) &:= \sum_{k\in\bN} 1_{(S^{(k)}_{[\eta]}, T^{(k)}_{[\eta]} )} (t) \frac{\eta \alpha }{\varepsilon}\left[ f^\prime(U^{b^\ast}_t) \right]^-, \\
Z_t^{(2)} (\eta) &:= \sum_{k\in\bN} 1_{(S^{(k)}_{[\eta]}, T^{(k)}_{[\eta]} )} (t)\frac \alpha \varepsilon \left[ R^{(q)}\rbra{
\inf_{s \in [t, t+\eta]}U^{b^\ast}_s} -\beta \right]^-.
\end{align}

It is sufficient to prove that $Z_t^{(i)} (\eta) \to 0$ for $i=1,2$ as $\eta \downarrow 0$ for Lebesgue-a.e.\ $t \geq 0$.
First, we have 
\begin{align}
 0 \leq 
 Z_t^{(1)} (\eta)
\leq 
 \frac{\eta \alpha }{\varepsilon}\left[f^\prime(U^{b^\ast}_t) \right]^-  \xrightarrow{\eta \downarrow 0} 0. \label{31}
\end{align}
For the convergence of $Z_t^{(2)}$, it suffices to show for $t \geq 0$ at which $X$ does not jump because $X$ does not jump at a.e.-$t$ and neither does $U^{b^*}$. Thus, we fix $t$ at which $X$ does not jump.
\begin{enumerate}
\renewcommand{\labelenumi}{(\arabic{enumi})}
\item Suppose that $U^{b^\ast}_t> b^\ast$. 
Then, for small enough $\eta>0$, because $X$ is c\'adl\'ag (and so is $U^{b^*}$) we have $\inf_{s \in [t, t+\eta]}U^{b^\ast}_s >b^\ast$ and thus $ R^{(q)}\rbra{\inf_{s \in [t, t+\eta]}U^{b^\ast}_s} -\beta\geq R^{(q)}(b^*)-\beta= 0$ by  the fact that $R^{(q)}(b^*)=\rho(b^*)$ and Lemma \ref{lemma_about_b_star}(ii). Thus, $Z_t^{(2)} (\eta) \xrightarrow{\eta \downarrow 0} 0$. 
\item Suppose that $U^{b^\ast}_t< b^\ast$.
Since we are assuming $X$ does not have a jump at $t$ and because $X$ is c\'adl\'ag, for small $\eta>0$, we have 
$\sup_{s \in [t-\eta , t+\eta]}U^{b^\ast}_s<b^\ast$, which implies that $t\not\in(S^{(k)}_{[\eta]} , T^{(k)}_{[\eta]} )$ for all $k\in\bN$. Thus, $Z_t^{(2)} (\eta) \xrightarrow{\eta \downarrow 0} 0$. 
\item Suppose that $U^{b^\ast}_t= b^\ast$. Then, by Remark \ref{30}(iii), since $R^{(q)}$ is continuous by the dominated convergence theorem and Lemma \ref{lemma_property_diff_process}(a) and since the process $U^{b^\ast}$ has right continuous paths, $Z_t^{(2)} (\eta) \xrightarrow{\eta \downarrow 0} 0$. 
\end{enumerate}
This completes the proof of (i).

(ii) Again, because $S^{(k)}_{[\eta]} +\eta = T^{(k)}_{[\eta]}$ and $R^{(q)}$ is non-decreasing,
\begin{align}
\begin{split}
\delta_{1}(\eta) 
&=\sum_{k\in\bN}\bE_x\sbra{ \frac{\varepsilon^{(k)}_{[\eta]} }{\varepsilon}
\int_{S^{(k)}_{[\eta]}}^{T^{(k)}_{[\eta]}} \rbra{  e^{-qt}\left[ f^\prime(U^{[\varepsilon], b^\ast}_t)\right]^+ 
+ \frac{1}{\eta}
e^{-qT^{(k)}_{[\eta]}}\rbra{R^{(q)} \left(U^{[\varepsilon], b^\ast}_{T^{(k)}_{[\eta]}} \right) -\beta} } \diff t }\\
&\leq  \bE_x\sbra{
\int_0^\infty  e^{-qt} \bar{G}_t(\eta) \diff t }, \end{split} \label{33}
\end{align}
where
\begin{align}
\bar{G}_t(\eta) &:=  \sum_{k\in\bN} 1_{(S^{(k)}_{[\eta]}, T^{(k)}_{[\eta]} )} {(t)}\rbra{\frac{\varepsilon^{(k)}_{[\eta]} }{\varepsilon} \left[ f^\prime(U^{[\varepsilon], b^\ast}_t) \right]^+ + \frac{1}{\varepsilon}\frac{\varepsilon^{(k)}_{[\eta]} }{\eta} \left[ R^{(q)}\left( 
\sup_{s \in [t, t+\eta]}U^{[\varepsilon], b^\ast}_s \right) -\beta \right]^+ } \\ &\geq 0.
\end{align}
By the same argument as (i) with \eqref{19}, \eqref{20} and \eqref{21} with $G_t$, $[\cdot]^-$ and infimum replaced by  $\bar{G}_t$, $[\cdot]^+$ and supremum, respectively, we can again use the reverse Fatou's lemma to obtain
\begin{multline}
\varlimsup_{\eta \downarrow 0}
\bE_x\sbra{
\int_0^\infty  e^{-qt} {\bar{G}_t(\eta)}
\diff t } \leq
\bE_x\sbra{
\int_0^\infty  e^{-qt} \varlimsup_{\eta \downarrow 0} {\bar{G}_t(\eta)}
\diff t }  \leq
\bE_x\sbra{
\int_0^\infty  e^{-qt} \varlimsup_{\eta \downarrow 0}
(\overline{Z}_t^{(1)} (\eta) + \overline{Z}_t^{(2)} (\eta))
\diff t },
\label{32}
\end{multline}
where
\begin{align}
\overline{Z}_t^{(1)} (\eta) &:= \sum_{k\in\bN} 1_{(S^{(k)}_{[\eta]}, T^{(k)}_{[\eta]} )} (t) \frac{\eta \alpha }{\varepsilon} [f^\prime(U^{[\varepsilon], b^\ast}_t) ]^+, \\
\overline{Z}_t^{(2)} (\eta) &:= \sum_{k\in\bN} 1_{(S^{(k)}_{[\eta]}, T^{(k)}_{[\eta]} )} (t) \frac{\alpha}{\varepsilon} \left[ R^{(q)}\rbra{
\inf_{s \in [t, t+\eta]}U^{[\varepsilon], b^\ast}_s} -\beta \right]^+.
\end{align}
As in the case of \eqref{31}, $\overline{Z}_t^{(1)} (\eta) \xrightarrow{\eta \downarrow 0} 0$ for all $t \geq 0.$ For the convergence of $\overline{Z}_t^{(2)} (\eta) $, results analogous to (1)-(3) in (i) for $U^{[\varepsilon], b^\ast}$ instead of $U^{b^\ast}$ clearly hold (with replacing $[\cdot]^-$ with $[\cdot]^+$, infimum with supremum,  ``$>b^\ast$" with  ``$<b^\ast$'' in (1), ``$<b^\ast$'' with ``$>b^\ast$'' in (2).) Thus, $\overline{Z}_t^{(2)} (\eta) \xrightarrow{\eta \downarrow 0} 0$ for Lebesgue-a.e.\ $t \geq 0$ (at which $t$ does not jump). This completes the proof of (ii) in view of \eqref{33}.
\end{proof}

From \eqref{35} with $\eta \downarrow 0$ and applying Lemma \ref{lemma_delta_2_limit}, we have, for $x\in\bR$, 
\begin{multline}
\bE_x \sbra{\int_0^\infty e^{-qt}f^\prime(U^{b^\ast}_t )\diff t}
\leq\frac{v_{b^\ast}(x+\varepsilon)-v_{b^\ast}(x)}{\varepsilon} \\ \leq \bE_x \sbra{\int_0^\infty e^{-qt}f^\prime(U^{[\varepsilon], b^\ast}_t )\diff t} 
\leq \bE_x \sbra{\int_0^\infty e^{-qt}f^\prime(U^{b^\ast}_t +\varepsilon)\diff t},  
\label{37}
\end{multline}
where in the last inequality we used Lemma \ref{lemma_property_diff_process}(a). 
By taking the limit as $\varepsilon \downarrow 0$, with Lemma \ref{Lem302}, \eqref{36} and the dominated convergence theorem, the right-hand derivative becomes
\begin{align}
v_{b^\ast}^{\prime, +}(x)=\bE_x \sbra{\int_0^\infty e^{-qt}f^\prime(U^{b^\ast}_t )\diff t}, \quad x\in\bR.
\end{align}
The function $v_{b^\ast}$ is continuous on $\bR$ by \eqref{37} and its right derivative is non-decreasing. 
Thus, by \cite[Theorem 6.4]{HirLem2001}, the function $v_{b^\ast}$ is convex.


For $\varepsilon \in \mathbb{R}$, we have $|U^{[\varepsilon],b^\ast}_t - U^{b^\ast}_t| \leq |\varepsilon|$ by Lemma \ref{lemma_property_diff_process}(a) and \cite[Lemma D.1]{NOBA2023174}.
Because $f'$ is non-decreasing and continuous, monotone convergence gives  
\begin{align}
v_{b^\ast}^{\prime, +}(x+\varepsilon)-v_{b^\ast}^{\prime, +}(x)=\bE_x \sbra{\int_0^\infty e^{-qt}\left(f^\prime(U^{[\varepsilon],b^\ast}_t ) - f^\prime(U^{b^\ast}_t ) \right) \diff t} \xrightarrow{\varepsilon \to 0} 0, \quad x\in\bR. \label{jlk}
\end{align}
This,  together with the convexity of $v_{b^\ast}$, shows that $v_{b^\ast}^{\prime, +}$ is in fact a derivative and is continuous.
%
This completes the proof of Theorem \ref{density_v_b_star} when $b^* \in \R$.

%


\subsection{For the case $b^* = -\infty, +\infty$}

Suppose $b^* = - \infty$. Then, for $\varepsilon\in\R$,
\begin{align}
\frac{v_{b^*}(x+\varepsilon)-v_{b^*}(x)}{\varepsilon}=\E\left[\int_0^{\infty}e^{-qt}\frac{f(x+X^{(\alpha)}_t+\varepsilon)-f(x+X^{(\alpha)}_t)}{\varepsilon}\diff t\right]& \\ \xrightarrow{\varepsilon \downarrow 0} \E_x&\left[\int_0^{\infty}e^{-qt}f'(X^{(\alpha)}_t)\diff t\right],
\end{align}
where the limit holds by monotone convergence. The proof for the case $b^*=+\infty$ is similar, with $X^{(\alpha)}$ replaced with $X$.
%

{As in the case of $b^* \in (-\infty, \infty)$, the right-hand derivative is in fact a derivative and it is continuous.  In addition, the derivative is monotone and $v_{b^*}$ is thus convex.}

%


\appendix
\section{Proof of Lemma \ref{Lem402}}\label{app_lemma}


We only  provide the proof for the case $b^{*} \in (-\infty, \infty)$; the cases $b^* = \pm \infty$  are simplier and follow verbatim by replacing the resolvent measure \eqref{def_resolvent} with those for \lev processes $X$ and $X^{(\alpha)}$.

Suppose $X$ satisfies {that} $\Pi (0, \infty) <\infty $.  This is Case 1 in Section \ref{SubSec202} because $X$ is of unbounded variation.
Then, we can write it as 
\begin{align}
	X_t ={X^{\textrm{SN}}_t} + \sum_{n=1}^{N_t} J_n, \quad t \geq 0,
\end{align} 
where ${X^{\textrm{SN}}_t}= (X^{\textrm{SN}}_t)_{t\geq 0}$ is a spectrally negative L\'evy process with unbounded variation paths, $N=(N_t)_{t\geq 0}$ is {an} independent Poisson process with intensity $\Pi(0, \infty)$ and ${\{ J_n \}}_{ n \in \N}$ is a sequence of i.i.d.\ random variables {with} distribution $\Pi(\cdot \cap (0, \infty))/\Pi(0, \infty)$. 
{For $\theta \geq 0$,} let $W^{(\theta)}: \R \to \R$ {denote the} $\theta$-scale function of ${X^{\textrm{SN}}}$ and  $\mathbb{W}^{(\theta)}: \R \to \R$ {denote the} $\theta$-scale function of the drift-changed spectrally negative \lev process $(X^{\textrm{SN}}_t - \alpha t)_{t \geq 0}$. 
For the definition and important properties of the scale functions, see, e.g., \cite[Section 8]{Kyp2014}. 
In particular, it is known as in Lemma 8.2 of \cite{Kyp2014} that $W^{(\theta)}$ and  $\mathbb{W}^{(\theta)}$ are continuous on $\mathbb{R}$ and $C^1$ on $\mathbb{R} \backslash \{0\}$ because $X$ is assumed to be of unbounded variation.


Let ${U^{\textrm{SN}, b^*}}=(U^{\textrm{SN}, b^*}_t)_{t\geq 0}$ be the {associated} refracted L\'evy process of $X^{\textrm{SN}}$ at $b^*$, i.e. the solution to the SDE
\[
{U^{\textrm{SN}, b^*}_t}= {X^{\textrm{SN}}_t} - \alpha\int_0^t1_{(b^*,\infty)}(U_s^{\textrm{SN},b^*})\diff s,\qquad t \geq 0.
\]
As shown in \cite{KypLoe2010}, the solution exists and is unique. 

In addition, we define, for any $g: \mathbb{R} \to \mathbb{R}$, 
\begin{align}
	{H^{\textrm{SN}, (b^*, q)}_g}(x):=\E_x \left[\int_0^{\infty}e^{-q t} g({U^{\textrm{SN}, b^*}_t)} \diff t \right], \quad q>0 , x\in\bR,  
\end{align} 
which, by \cite[Theorem 6]{KypLoe2010}, can be written, if $|H^{\textrm{SN}, (b^*, q)}_g(x)|<\infty$,
as
\begin{align}
	\E_x \left[\int_0^{\infty }e^{-q t} g({U^{\textrm{SN}, b^*}_t)} \diff t \right]
	= \int_{-\infty}^\infty g(y)R^{(q)}(x,y)\diff y, 
\end{align}
where, for all $x, y \in\R$,
\begin{align}\label{def_resolvent}
\begin{split}
R^{(q)}(x,y):&=\left(e^{\Phi(q)(x-b^\ast)}+\alpha\Phi(q)e^{-\Phi(q)b^*}1_{\{x>b^\ast \}}\int_{b^*}^xe^{\Phi(q)z}\mathbb{W}^{(q)}(x-z)\diff z\right)\\& \qquad \qquad \times \frac{\varphi(q)-\Phi(q)}{\alpha\Phi(q)}e^{\varphi(q)b^*}\int_{b^*}^{\infty}e^{-\varphi(q)z}W^{(q)\prime}(z-y)\diff z \\
&-\left(W^{(q)}(x-y)+\alpha1_{\{x>b^*\}}\int_{b^*}^x\mathbb{W}^{(q)}(x-z)W^{(q)\prime}(z-y)\diff z\right).
\end{split}
\end{align}
Here, $\Phi$ and $\varphi$ are right inverses of Laplace exponents of $X^{\text{SN}}$ and $(X^{\textrm{SN}}_t - \alpha t)_{t \geq 0}$, respectively (see, e.g., \cite[Section 8]{Kyp2014}).

Define for $x\in\R$ and $q>0$, $H^{(b^*, q)}_{f'}(x)=\E_x\left[ \int_0^\infty e^{-qt}f'(U_t^{b^*})\diff t\right]$, where $U^{b^*}=(U_t^{b^*})_{t\geq0}$ is the refracted process of $X$ at $b^*$. Thus, by an application of the Markov property we have 
\begin{align}
&v_{b^*}'(x) = H^{(b^*, q)}_{f'}(x) \\
	&={H^{\textrm{SN}, (b^*, q)}_{f'}}(x) -\E_x \left[e^{-q T^{N}}{H^{\textrm{SN}, (b^*, q)}_{f'}} (U^{\textrm{SN}, b^*}_{T^{N}})\right]
+\E_x  \left[ e^{-q T^{N}}{H^{(b^*, q)}_{f'}(U^{\textrm{SN}, b^*}_{T^{N}}+J_1)} \right]\\
	&={H^{\textrm{SN}, (b^*, q)}_{f'}}(x)
	-\Pi(0, \infty) {H^{\textrm{SN}, (b^*, \hat{q})}_{H^{\textrm{SN}, (b^*, q)}_{f'}}}(x)
	+\int_{(0, \infty)} {H^{\textrm{SN}(b^*, \hat{q})}_{H^{ (b^*, q)}_{f'}(\cdot + u) }}(x)\Pi(\diff u), 
\end{align}
where $T^{N}$ is the first jump time of the process $N$ and $\hat{q}:=q+\Pi(0, \infty)$.
Because $W^{(q)}(0)=0$ (as $X$ is of unbounded variation),
for a function $g\in C^1 (\bR)$, we have 
\begin{align}
	{H^{\textrm{SN}, ({b^\ast}, q)}_g} (x)=
	\int_{b^\ast}^x \int_{{-\infty}}^\infty g(y) \partial_1R^{(q)}(z,y)\diff y \diff z+{H^{\textrm{SN}, ({b^\ast}, q)}_g} (b^\ast), \quad x \in \R,
\end{align}
where 
\begin{align}
\partial_1R^{(q)}(x,y) &:=\left(\Phi(q)e^{\Phi(q)(x-b^*)}+\alpha\Phi(q)e^{-\Phi(q)b^*}1_{\{x>b^*\}}\int_{b^*}^xe^{\Phi(q)z}\mathbb{W}^{(q)\prime}(x-z)\diff z\right)\notag\\& \qquad \qquad \times \frac{\varphi(q)-\Phi(q)}{\alpha\Phi(q)}e^{-\varphi(q)b^*}\int_{b^*}^{\infty}e^{-\varphi(q)z}W^{(q)\prime}(z-y)\diff z\notag\\
&-\left(W^{(q)\prime}(x-y)+\alpha1_{\{x>b^*\}}\int_{b^*}^x\mathbb{W}^{(q)\prime}(x-z)W^{(q)\prime}(z-y)\diff z\right).
\end{align}
Hence, the function $v_{b^*}'(x)$ has a density 
\begin{align}\label{sec_der_sn}
\begin{split}
	v_{b^*}''(x):=&
	\int_{-\infty}^\infty f'(y) \partial_1R^{(q)}(x,y)\diff y-\Pi(0, \infty)	\int_{-\infty}^\infty {H^{\textrm{SN}, (b^*, q)}_{f'}}(y) \partial_1R^{(\hat{q})}(x,y)\diff y\\
	&+	\int_{-\infty}^\infty \int_{(0,\infty)} {H^{ (b^*, q)}_{f'}}(y+u)\Pi(\diff u) \partial_1R^{(\hat{q})}(x,y)\diff y
	,\qquad x\in\R.
\end{split}
\end{align}

The right-hand side of \eqref{sec_der_sn} is well-defined by the following lemma.

\begin{lemma} The following holds for any $q > 0$.
(i) For any non-decreasing function $g$ of polynomial growth, we have
\begin{align}\label{integrability}
\int_{-\infty}^{\infty}|g(y)|\partial_1R^{(q)}(x,y) \diff y <\infty.
\end{align}
(ii) $y \mapsto  \int_{(0,\infty)} {H^{ (b^*, q)}_{f'}}(y+u)\Pi(\diff u)$ is of polynomial growth.
\end{lemma}

\begin{proof}
(i) 
We have
\begin{align}
&\partial_1R^{(q)}(x,y) \\
&=\Phi(q)e^{\Phi(q)(x-b^\ast)}\frac{\varphi(q)-\Phi(q)}{\alpha\Phi(q)} e^{\varphi(q)b^*}\int_{b^*}^{\infty}e^{-\varphi(q)z}W^{(q)\prime}(z-y)\diff z-W^{(q)\prime}(x-y)\notag\\
&+\alpha\Phi(q)e^{-\Phi(q)b^*}1_{\{x>b^\ast \}}\int_{b^*}^xe^{\Phi(q)z}\mathbb{W}^{(q)\prime}(x-z)\diff z \\
& \qquad \qquad  \qquad  \times \frac{\varphi(q)-\Phi(q)}{\alpha\Phi(q)} e^{\varphi(q)b^*}\int_{b^*}^{\infty}e^{-\varphi(q)z}W^{(q)\prime}(z-y)\diff z\notag\\
&-\alpha1_{\{x>b^*\}}\int_{b^*}^x\mathbb{W}^{(q)\prime}(x-z)W^{(q)\prime}(z-y)\diff z\notag\\
&=A^{(q)}(x,y)+\alpha\int_{b^*}^x\mathbb{W}^{(q)\prime}(x-z)A^{(q)}(z,y) \diff z,
\end{align}	
where
\begin{align}
&A^{(q)}(x,y) \\ &:=\Phi(q)e^{\Phi(q)(x-b^\ast)}\frac{\varphi(q)-\Phi(q)}{\alpha\Phi(q)} e^{\varphi(q)b^*} \int_{b^*}^{\infty}e^{-\varphi(q)z}W^{(q)\prime}(z-y)\diff z-W^{(q)\prime}(x-y)\notag\\
&=\Phi(q)e^{\Phi(q)(x-b^\ast)}\left(\frac{\varphi(q)-\Phi(q)}{\alpha\Phi(q)} e^{\varphi(q)b^*}\int_{b^*}^{\infty}e^{-\varphi(q)z}W^{(q)\prime}(z-y)\diff z-W^{(q)}(b^*-y)\right)\notag\\
&+\Phi(q)(e^{\Phi(q)(x-b^\ast)}W^{(q)}(b^*-y)-W^{(q)}(x-y))-(W^{(q)\prime}(x-y)-\Phi(q)W^{(q)}(x-y))\notag\\
&=\Phi(q) \left[e^{\Phi(q)(x-b^\ast)}R^{(q)}(b^*,y)+ 
(u^{(q)}(x-y)-e^{\Phi(q)(x-b^\ast)}u^{(q)}(b^*-y)) 
-v^{(q)}(x,y) \right].
\end{align}
Here $u^{(q)}$ is the $q$-resolvent density of $-X^{\textrm{SN}}$ (see Theorem 2.7 of \cite{KuzKypRiv2012}) by \cite[Remark 3.3]{HerPerYam2016}  and  $v^{(q)}(x, y) := W^{(q)\prime}(x-y)/\Phi(q)-W^{(q)}(x-y)$ is the $q$-resolvent measure of $-\underline{X}^{\textrm{SN}}$ by \cite[Corollary 2.2]{KuzKypRiv2012} where $\underline{X}^{\textrm{SN}}$ is the running infimum of $X^{\textrm{SN}}.$

Let $g^+$ and $g^-$ be the positive and negative parts of $g$ so that $|g| = g^+ + g^-$.

By \eqref{def_resolvent} and also the facts that $u^{(q)}$ and $v^{(q)}$ are $q$-resolvent measures,
\begin{align}
\int_{\R}g^+(y)R^{(q)}(b^*,y)\diff y&=\mathbb{E}_{b^*}\left[\int_0^{\infty}e^{-qt}g^+(U^{\textrm{SN}, b^*}_t)\diff t\right], \\
\int_{\R}g^+(y)
v^{(q)}(x,y)
\diff y&=
\mathbb{E}_{x}\left[\int_0^{\infty}e^{-qt}g^+(-\underline{X}^{\textrm{SN}}_t)\diff t\right], 
\end{align}
and 
\begin{multline}
\int_{\R}g^+(y)(u^{(q)}(x-y)-e^{\Phi(q)(x-b^\ast)}u^{(q)}(b^*-x)) \diff y \\=\mathbb{E}_{x}\left[\int_0^{\infty}e^{-qt}g^+(-X^{\textrm{SN}}_t)\diff t\right] -e^{\Phi(q)(x-b^\ast)}\mathbb{E}_{b^*}\left[\int_0^{\infty}e^{-qt}g^+(-X^{\textrm{SN}}_t)\diff t\right],
\end{multline}
all of which are finite by Assumptions  \ref{Ass301} and \ref{assump_f} (see also Lemma \ref{Lem302}). Substituting these,
	\begin{multline}\label{decom_1}
\int_{-\infty}^{\infty}g^+(y)\partial_1R^{(q)}(x,y) \diff y=(r_1^{(q)}(x)+r_2^{(q)}(x))+\alpha \int_{b^*}^x\mathbb{W}^{(q)\prime}(x-z)(r_1^{(q)}(z)+r_2^{(q)}(z))\diff z,
	\end{multline}
where for $x\in\R$ 
\begin{align}
r^{(q)}_1(x)&=\Phi(q)e^{\Phi(q)(x-b^*)}\mathbb{E}_{b^*}\left[\int_0^{\infty}e^{-qt}\left(g^+(U^{\textrm{SN}, b^*}_t)-g^+(-X^{\textrm{SN}}_t)\right)\diff t \right] \\ &-
\Phi(q)\mathbb{E}_{x}\left[\int_0^{\infty}e^{-qt}g^+(-\underline{X}^{\textrm{SN}}_t)\diff t\right],\notag\\
r^{(q)}_2(x)&=\Phi(q)\E_{x}\left[\int_0^{\infty}e^{-qt}g^+(-X^{\textrm{SN}}_t)\diff t\right].
\end{align}
By the monotonicity of $g^+$, the mapping $x\mapsto r^{(q)}_1(x)$ (resp. $x\mapsto r^{(q)}_2(x)$) is non-decreasing  (resp. non-increasing), therefore by \cite[Remark 3.3]{HerPerYam2016} we obtain, for $x \in \R$,
\begin{align}
\int_{-\infty}^{\infty}g^+(y)\partial_1R^{(q)}(x,y) \diff y\leq (r_1^{(q)}(x)+r_2^{(q)}(x))+\alpha (r_1^{(q)}(x)+r_2^{(q)}(b))\mathbb{W}^{(q)}(x-b^*). 
\end{align}
Using Lemma \ref{Lem302} together with \cite[Exercise 7.1(ii)]{Kyp2014} gives that 
$r_i^{(q)}(z)<\infty$ for $z\in\R$ and $i=1,2$, and hence
 $\int_{-\infty}^{\infty}g^+(y)\partial_1R^{(q)}(x,y) \diff y<\infty$. By a similar argument,  $\int_{-\infty}^{\infty}g^-(y)\partial_1R^{(q)}(x,y) \diff y<\infty$ and hence, we obtain \eqref{integrability}.


(ii) Because ${H^{\textrm{SN}, (b^*, q)}_{f'}}$ is of polynomial growth by the proof of \cite[Lemma 3]{NobYam2022}, for some $k_1$, $k_2$ and $N$, 
\begin{multline}
\left|\int_{(0,\infty)} {H^{ (b^*, q)}_{f'}}(y+u)\Pi(\diff u)\right| \leq \int_{(0,\infty)} |{H^{ (b^*, q)}_{f'}}(y+u)|\Pi(\diff u)\\ \leq \int_{(0,\infty)}\left(k_1+k_2|u+y|^N\right)\Pi(\diff u) \leq k_1 \Pi(0,\infty)+k_2\sum_{n=0}^N \binom N n 
\int_{(0,\infty)}u^{N-n}\Pi(\diff u)y^{n},
\end{multline}
which implies that $\int_{(0,\infty)} {H^{ (b^*, q)}_{f'}}(y+u)\Pi(\diff u)$ is of polynomial growth. 
\end{proof}

Therefore, $v_{b^*}''(x)$ defined by \eqref{sec_der_sn}, is well-defined for every $x\in\R$.
In addition, by dominated convergence, $v_{b^*}''(x)$ is continuous.

\par In the same way, we can prove that the case with $\Pi(-\infty , 0)<\infty$ satisfies Assumption \ref{Ass401a}.

\bibliographystyle{plain}
\bibliography{NOBA_references_05} 

\end{document}